\documentclass[a4paper,11pt,3p]{article}
\usepackage{graphicx} 
\usepackage{amsmath,amsthm,amscd,amsfonts,amssymb,enumitem,algorithm2e,mathtools,mathrsfs}
\usepackage{latexsym,xcolor}
\usepackage{tcolorbox, lmodern, array}
\usepackage[colorlinks=true, linkcolor=black, anchorcolor=black, citecolor=black, filecolor=black, urlcolor=black]{hyperref}
\usepackage{algorithm2e, float, caption}
\usepackage{graphicx}
\usepackage{stmaryrd}
\usepackage{wrapfig}
\usepackage{relsize}
\usepackage{xcolor}
\newtheorem{thm}{Theorem}[section]
\newtheorem{lemma}[thm]{Lemma}

\newtheorem{rem}[thm]{Remark}

\numberwithin{equation}{section}
\newcommand\blfootnote[1]{%
	\begingroup 
	\renewcommand\thefootnote{}\footnote{#1}%
	\addtocounter{footnote}{-1}%
	\endgroup
}
\newcommand\restr[2]{{
  \left.\kern-\nulldelimiterspace 
  #1 
  \littletaller 
  \right|_{#2} 
  }}

\newcommand{\littletaller}{\mathchoice{\vphantom{\big|}}{}{}{}}

\renewcommand{\epsilon}{\varepsilon}
\title{ Analysis of an inexact Schwarz domain decomposition method in 3D }
\author{Nibedita Ghosh\thanks{Email id: ghosh@igpm.rwth-aachen.de} \\
		Institute for Geometry and Practical Mathematics\\ RWTH Aachen University, 52072 Aachen, Germany}
\date{}
\begin{document}
\maketitle

\begin{abstract}
We consider an inexact version of the Schwarz domain decomposition method in $3$D to solve a Laplace problem in a domain that is the union of two overlapping spheres. We solve each subproblem approximately by projecting its boundary data $L^2$-orthogonally onto the space of spherical harmonics of degree at most $N$. The model problem arises from the ddCOSMO solver used for implicit solvation in computational chemistry. We analyze the Schwarz-spherical domain decomposition method by establishing a new variant of the maximum principle.
\end{abstract}


\section{Introduction}\label{sec:1}
Implicit (continuum) solvation models are an established tool in computational chemistry for describing the effect of a solvent on a solute molecule, see \cite{lipparini2016perspective, tomasi2005quantum}. Instead of resolving the solvent molecules explicitly, these models represent the solvent as a continuous dielectric medium that surrounds a molecular cavity containing the solute. In many models, this cavity is the union of overlapping balls, each of which is centered at one of the atoms of the solute. In the conductor-like screening model (COSMO) \cite{klamt1993cosmo}, the computation of the electrostatic solvation energy requires the solution of a Laplace equation in this union of balls, with Dirichlet boundary data given by the electrostatic potential of the solute.

Canc\`es et al. proposed to solve the Laplace problem by a Schwarz domain decomposition method in which every ball is a subdomain in the paper \cite{cances2013domain}. The decisive feature of this approach, known as ddCOSMO, is that harmonic functions in a ball have an explicit representation in terms of spherical harmonics. The local Dirichlet problems are therefore discretized by approximating the boundary data on each sphere in the space of spherical harmonics up to a fixed degree $N$, and the corresponding local solutions are available in closed form. This results in a very efficient solver, in particular for large molecules; we refer to \cite{stamm2019make} for a practical guide and to \cite{quan2018domain} for an extension to the polarizable continuum model.

The convergence analysis of Schwarz methods has a long history. For overlapping subdomains and exact subdomain solvers, a classical approach based on the maximum principle is proposed in \cite{lionsschwarz,lions1988schwarz}. For the general theory of Schwarz methods and subspace correction methods, we refer to \cite{toselli2004domain, xu1992iterative}. Motivated by ddCOSMO, Ciaramella and Gander \cite{ciaramella2017analysis, ciaramella2018analysis, ciaramella2018analysis3} analyzed the parallel Schwarz method for chains of fixed-sized subdomains with a growing number of subdomains. Reusken and Stamm \cite{reusken2021analysis} studied the Schwarz method on domains that consist of many overlapping balls of fixed size, as they arise in COSMO, and showed how its rate of convergence depends on local and global geometric descriptors of the domain. Corso et al. derived trace estimates for harmonic functions along circular arcs and applied them to the convergence analysis of the Schwarz method on two overlapping disks in \cite{carvalho2026trace}, including a discrete variant of the method used in COSMO.

Apart from the two-dimensional analysis in \cite{carvalho2026trace}, these works assume that the subdomain problems are solved exactly. In ddCOSMO, however, the local boundary data are replaced by their approximations in a space of spherical harmonics, and this approximation interacts with the Schwarz iteration. Recently, Reusken \cite{reusken2026analysis} analyzed this effect for a two-dimensional model problem, the Schwarz-Fourier method on two overlapping disks, in which the boundary data of each local problem are projected onto a space of trigonometric polynomials. The key ingredient of that analysis is a new variant of the maximum principle for the inexact local solver, which leads to bounds for the contraction number in the maximum norm.

In this paper, we extend this approach to three dimensions, which is the setting relevant for ddCOSMO. We consider the Laplace problem on the union of two overlapping spheres and the Schwarz-spherical domain decomposition method, in which the boundary data of each local problem are projected $L^2$-orthogonally onto the space $\mathcal{S}_N$ of spherical harmonics of degree at most $N$. The main difficulty is that the maximum principle does not hold for the inexact local solver. The solution of the exact local problem is given by the Poisson integral, and the maximum principle follows from the nonnegativity of the Poisson kernel together with its unit mean. For the inexact local solver, the Poisson kernel is replaced by its truncated expansion in Legendre polynomials. This truncated kernel $K_N$ oscillates and takes negative values close to the boundary sphere, which is a manifestation of the Gibbs phenomenon. In contrast to the two-dimensional case, where the truncated kernel is a trigonometric polynomial, the kernel in three dimensions is a sum of Legendre polynomials $L_n$ with coefficients $(2n+1)r^n$, in which the factor $2n+1$ grows with $n$. Hence, the analysis of the sign of $K_N$ requires different arguments.

The main contributions of this paper are as follows.
\begin{itemize}
    \item We formulate the Schwarz-spherical method as a fixed point iteration for the traces on the interior boundaries, which is governed by inexact Dirichlet-to-Dirichlet maps $\Lambda_{i,N}$, $i=1,2$. The total error splits into a discretization error and an iteration error, and we study the iteration error.
    \item We show that the inexact local solution operator $\Delta^{-1}P_N$ is an integral operator with the truncated Poisson kernel $K_N(r\eta,\xi)=\sum\limits_{n=0}^{N}(2n+1)r^nL_n(\eta\cdot\xi)$, which is the $L^2$-orthogonal projection of the Poisson kernel onto $\mathcal{S}_N$ (Lemma~\ref{lemma:fmv}).
    \item Our main result, Theorem~\ref{thm:main}, states that for every odd $N\geq 5$ the truncated kernel satisfies $K_N(r\eta,\xi)\geq 0$ for all $\eta,\xi$ on the unit sphere and all $0\leq r\leq r_N$, with the explicit radius
    \begin{align*}
        r_N=1-\frac{2\ln{(2(N+1))}}{N+1}.
    \end{align*}
    The proof relates the minimum of $K_N(r\eta,\cdot)$ to the antipodal point $\xi=-\eta$, where $K_N$ reduces to an alternating sum with a closed form, and estimates the zeros of this closed form by means of the Lambert $W$ function \cite{hoorfar2008inequalities, hassani2012approximation}.
    \item As a consequence, we obtain a new variant of the maximum principle (Theorem~\ref{thm:nmp}): for piecewise continuous boundary data $f$, the inexact local solution $\Delta^{-1}P_Nf$ is bounded from below and above by the minimum and the maximum of $f$ in the concentric ball of radius $r_N$, that is, everywhere except in a boundary layer of width $\delta_{th}(N)=1-r_N=\mathcal{O}(\ln N/N)$.
    \item In numerical experiments for $N\leq 100$, we compute the largest radius $r_{num}(N)$ for which $K_N$ is nonnegative. For odd $N$, the theoretical width $\delta_{th}(N)$ exceeds the computed width $\delta_{num}(N)=1-r_{num}(N)$ by less than a factor of two, and the quotient $\delta_{num}(N)/\delta_{th}(N)$ tends to increase with $N$. This indicates that the order $\mathcal{O}(\ln N/N)$ is sharp.
    \item We use the new maximum principle to bound the inexact Dirichlet-to-Dirichlet map $\Lambda_{1,N}$ in the maximum norm, which determines the convergence of the Schwarz-spherical method.
\end{itemize}

The paper is organized as follows. In Section~\ref{sec:2}, we recall the solution of the Laplace problem in the unit ball in terms of spherical harmonics and derive the Poisson kernel from its expansion in Legendre polynomials. Section~\ref{sec:3} summarizes the properties of the Poisson kernel that lead to the maximum principle. In Section~\ref{sec:4}, we describe the classical Schwarz method for two overlapping balls and its formulation in terms of Dirichlet-to-Dirichlet maps. The Schwarz-spherical method is introduced in Section~\ref{sec:5}, where we also derive the representation of the inexact local solver by the truncated kernel, prove the positivity of the kernel and the new maximum principle, present the numerical experiments, and apply the maximum principle to the convergence analysis.

\section{Solution of a Laplace problem in a unit ball in 3D}\label{sec:2}

We consider the Laplace problem in a unit ball $B=\{ x\in \mathbb{R}^3 \;|\; |x|<1 \}$ in $3$D   with piecewise continuous boundary data $f$ on $\partial B$. We parametrized the boundary of the ball by the angles $(\theta, \phi)\in [0, \pi]\times [0, 2\pi).$ We assume that the function $f$ has only finitely many discontinuities on $\partial B$ and the points of discontinuities are denoted by $D_f=\{ x_1, x_2, \cdots, x_p \}.$ We further assume that the left-hand and right-hand limits of the function at these points exist, that is $\lim_{(\theta,\phi)\uparrow x_i} f(\theta, \phi)=f_i^-$ and $\lim_{(\theta,\phi)\downarrow x_i} f(\theta, \phi)=f_i^+.$ We allow the case $D_f=\emptyset$. We denote the space of such functions by $\mathcal{C}_D$. We analyze the problem
  \begin{align}\label{eqn:le}
   \begin{cases}
    \Delta w &=0 \text{ in } B=\{ x\in \mathbb{R}^3 \;|\; |x|<1 \},\\
    w&= f \text{ on } \partial B.
    \end{cases}
  \end{align}
We use spherical coordinates $x=r\sin{\theta}\cos{\phi}, y=r\sin{\theta}\sin{\phi}$ and $z=r\cos{\theta}$, where $0\leq r \leq  1, 0\leq \theta \leq \pi, 0\leq \phi < 2\pi$ and the transformed Laplace equation becomes
\begin{align}\label{eqn:lep}
    \frac{1}{r^2}\frac{\partial}{\partial r}\left(r^2\frac{\partial w}{\partial r}\right)+\frac{1}{r^2\sin{\theta}}\frac{\partial}{\partial\theta}\left(\sin{\theta}\frac{\partial w}{\partial \theta}\right)+\frac{1}{r^2\sin^2{\theta}}\frac{\partial^2 w}{\partial\phi^2}=0
\end{align}
with the respective boundary condition
\begin{align*}
    w(1,\theta,\phi)=f(\theta,\phi), 0 \leq \theta \leq \pi, 0\leq \phi < 2 \pi.
\end{align*}
We use the separation of variables method. We insert
\begin{align*}
    w(r,\theta,\phi)=R(r)\Theta(\theta)\Phi(\phi)
\end{align*}
into the equation \eqref{eqn:lep} to derive
\begin{align*}
    \frac{1}{\Theta \sin{\theta}}\frac{d}{d\theta}\left(\sin{\theta}\frac{d\Theta}{d\theta}\right)+\frac{1}{\Phi\sin^2{\theta}}\frac{d^2\Phi}{d\phi^2}=-\frac{1}{R}\frac{d}{dr}\left(r^2\frac{dR}{dr}\right)=-\lambda.
\end{align*}
We then choose another separation constant $\lambda^{'}$ to get the ordinary differential equations as
\begin{align*}
    \frac{d}{dr}\left(r^2\frac{dR}{dr}\right)=\lambda R, \frac{d^2\Phi}{d\phi^2}+\lambda^{'}\Phi=0,\\
    \sin{\theta}\frac{d}{d\theta}\left(\sin{\theta}\frac{d\Theta}{d\theta}\right)+(\lambda\sin^2{\theta}-\lambda^{'})\Theta=0.
\end{align*}
We are looking for the bounded solution inside the sphere. Therefore, we discard the unbounded case and get the solution for the $R$ and $\Phi$ equations as
\begin{align}
    R(r)&=C_nr^n, \text{ where }\lambda = n(n+1), n\in \mathbb{N}\cup \{0\}, \\
    \Phi(\phi) &= C_m e^{im\phi}, \lambda^{'}=m^2, m\in \mathbb{Z}.
\end{align}
We obtain the equation for $\Theta$ as
\begin{align*}
    \frac{d^2\Theta}{d\theta^2}+\frac{\cos{\theta}}{\sin{\theta}}\frac{d\Theta}{d\theta}+\left(n(n+1)-\frac{m^2}{\sin^2{\theta}}\right)\Theta=0.
\end{align*}
We employ the change of variable $t=\cos{\theta}\in (-1,1)$ and achieve the generalized Legendre equation as
\begin{align}\label{eqn:gle}
    (1-t^2)\frac{d^2 \xi}{dt^2}-2t \frac{d\xi}{dt}+\left( n(n+1)-\frac{m^2}{1-t^2}\right) \xi = 0, \text{ where } \xi(t)=\Theta(\theta).
\end{align}
Therefore, the bounded solution of the equation \eqref{eqn:gle} is represented by the Legendre polynomial of degree $n$ and order $m$ and explicitly given by 
\begin{align*}
    L_n^m (t)= (1-t^2)^\frac{m}{2}\frac{d^m}{dt^m} (L_n(t)), |m|\leq n,
\end{align*}
where $L_n(t)$ is the orthogonal Legendre polynomial over $(-1,1)$ such that
\begin{align*}
    L_n(t)=\frac{1}{2^n n!}\frac{d^l}{dt^n}[(t^2-1)^n], L_n(1)=1 
\end{align*}
and
\begin{align*}
    \int_{-1}^{1} L_n(t)L_k(t)=0 \text{ for } n\neq k.
\end{align*}
Therefore, the solution of the equation \eqref{eqn:lep} is given by
\begin{align}\label{eqn:se}
    w(r,\theta, \phi)=\sum_{n=0}^{\infty}\sum_{m=-n}^{n} a_{nm}r^n L_n^m(\cos{\theta}) e^{im\phi},
\end{align}
where the coefficient $a_{lm}$ has to be determined by the boundary conditions.\\
\textbf{Spherical harmonics:} For $n\in \mathbb{N}\cup \{0\}$ and $|m|\leq n,$ the spherical harmonic function $Y^m_n$ of degree $m$ and order $n$ is defined by
\begin{align}\label{eqn:sh}
    Y^m_n(\theta, \phi)=\sqrt{\frac{2n+1}{4\pi}\frac{(n-m)!}{(n+m)!}}L_n^m(\cos{\theta})e^{im\phi}, \text{ for }(\theta, \phi)\in [0, \pi]\times [0, 2\pi).
\end{align}
Hence, we can rewrite the solution \eqref{eqn:se} in terms of the spherical harmonics as 
\begin{align}\label{eqn:sts}
    w(r,\theta,\phi)=\sum_{n=0}^{\infty}\sum_{m=-n}^{n} a_{nm} r^n Y_n^m(\theta, \phi).
\end{align}
Let the boundary data $f(\theta, \phi)\in L^2(\partial B)$. Now since $\{Y_n^m \;|\; n\in \mathbb{N}\cup \{0\}, |m|\leq n\}$ forms a complete orthonormal basis of the Hilbert space $L^2(\partial B_1(0))$, therefore, $f$ can be expressed as a Fourier-Laplace series
\begin{align}\label{eqn:bdata}
    f(\theta, \phi)=\sum_{n=0}^{\infty}\sum_{m=-n}^{n} a_{nm} Y^m_n(\theta,\phi),
\end{align}
where the coefficients are given by
\begin{align*}
    a_{nm}=\langle f, Y_n^m \rangle_{L^2(\partial B)}=\int_{0}^{\pi}\int_{0}^{2\pi}f(\theta,\phi)\overline{Y_n^m(\theta,\phi)}\sin{\theta}\; d\phi d\theta.
\end{align*}
We can consider the partial sum of the series \eqref{eqn:bdata} as
\begin{align*}
    S_N(\theta, \phi)=\sum_{n=0}^{N}\sum_{m=-n}^{n}a_{nm} Y^m_n(\theta,\phi).
\end{align*}
We have $f\in L^2(\partial B)$, therefore $\lim_{N\to\infty} S_N(\theta, \phi)=f(\theta, \phi)$ in $L^2(\partial B).$ We can also derive that
\begin{align*}
    \lim_{N\to\infty} S_N(\theta, \phi)&=f(\theta, \phi) \text{ for all } (\theta, \phi)\in \partial B\setminus D_f,\\
    \lim_{N\to\infty} S_N(x_i)&=\frac{1}{2}(f_i^{-}+f_i^{+}) \text{ for all } x_i\in D_f.
\end{align*}

Another representation of the solution of the Laplace equation \eqref{eqn:le} is given by the Poisson kernel in $3$D. We can denote any point on the ball $B$ as $x=(r\sin{\theta}\cos{\phi}, r\sin{\theta}\sin{\phi}, r\cos{\theta})=r\eta $, where $\eta=(\sin{\theta}\cos{\phi}, \sin{\theta}\sin{\phi}, \cos{\theta})\in \partial B$. We can rewrite the solution \eqref{eqn:sts} as
\begin{align*}
    w(x):=w(r\eta)&=\sum_{n=0}^{\infty}\sum_{m=-n}^{n} a_{nm}r^n Y_n^m(\eta)\\
    &=\sum_{n=0}^{\infty}\sum_{m=-n}^{n} \langle f, Y_n^m \rangle_{L^2(\partial B)} r^n Y_n^m(\eta)\\
    &=\sum_{n=0}^{\infty}\sum_{m=-n}^{n}\left(\int_{|\xi|=1}f(\xi)\overline{Y_n^m(\xi)}d\xi \right) r^n Y_n^m(\eta)\\
    &=\int_{|\xi|=1} f(\xi) \left(\sum_{n=0}^{\infty} r^n \sum_{m=-n}^{n} \overline{Y_n^m(\xi)}  Y_n^m(\eta)\right) d\xi.
\end{align*}
We have
\begin{align*}
    \sum_{m=-n}^{n} \overline{Y_n^m(\xi)}  Y_n^m(\eta) =\frac{(2n+1)}{4\pi}L_n(\xi\cdot\eta).
\end{align*}
This yields
\begin{align}\label{eqn:cs}
     w(r\eta)&=\frac{1}{4\pi}\int_{|\xi|=1} f(\xi)  \left(\sum_{n=0}^{\infty} r^n (2n+1) L_n(\xi\cdot\eta)\right) d\xi.
\end{align}
We can establish that $\lim_{r\uparrow 1}w(r\eta)=f(\eta),$ see \cite{hackbusch2017elliptic} for a detailed explanation. Consequently, we can see that $w(r\eta)$ solves the Laplace problem given by \eqref{eqn:le}.

Now the Legendre polynomials are generated by the function $g(x,t)=\frac{1}{\sqrt{1-2xt+t^2}}$. We can expand the function as a Taylor series in $t$ and derive
\begin{align*}
    \sum_{n=0}^{\infty} L_n(t) r^n=\frac{1}{\sqrt{1-2rt+r^2}}=G(r,t), \text{ where }|r|<1.
\end{align*}
We can differentiate both sides with respect to $r$ and compute
\begin{align*}
    G+2r\frac{\partial G}{\partial r}=\frac{1-r^2}{(1-2rt+r^2)^{\frac{3}{2}}}.
\end{align*}
We can simplify the left-hand side identity and obtain
\begin{align*}
    \sum_{n=0}^{\infty}(2n+1)L_n(t)r^n=\frac{1-r^2}{(1-2rt+t^2)^{\frac{3}{2}}}.
\end{align*}
Hence, we get
\begin{align*}
     w(r\eta)&=\frac{1}{4\pi}\int_{|\xi|=1} \frac{1-r^2}{(1-2r(\xi.\eta)+r^2)^{\frac{3}{2}}}f(\xi) d\xi.
\end{align*}
Hence, the Poisson kernel in $3D$ is
\begin{align}\label{eqn:pk}
    K(r\eta,\xi)=\frac{1-r^2}{(1-2r(\xi\cdot\eta)+r^2)^\frac{3}{2}}=\sum_{n=0}^{\infty}(2n+1)r^n L_n(\xi\cdot\eta).
\end{align}
We can write the expression in terms of $r, \theta, \phi$ as 
\begin{align*}
    K(r,\theta,\phi)=\frac{1-r^2}{(1-2r(\sin{\theta}\sin{\theta^{'}}\cos({\phi-\phi^{'}})+\cos{\theta}\cos{\theta^{'}})+r^2)^{\frac{3}{2}}},
\end{align*}
where $\eta=(\sin{\theta}\cos{\phi}, \sin{\theta}\sin{\phi}, \cos{\theta})$ and $\xi=(\sin{\theta^{'}}\cos{\phi^{'}}, \sin{\theta^{'}}\sin{\phi^{'}}, \cos{\theta^{'}}) \in \partial B.$

\section{Idea of maximum principle}\label{sec:3}
We know from the maximum principle
\begin{align}\label{eqn:mp}
    \min_{\xi\in\partial B} f \leq w(r\eta, \xi)\leq \max_{\xi\in\partial B} f \text{ for all } r\eta\in B.
\end{align}
The Poisson kernel satisfies the relation
\begin{align}
    K(r\eta,\xi) &\geq 0, \label{eqn:pe} \\
    \frac{1}{4\pi}\int_{|\xi|=1}K(r\eta, \xi)d\xi &=1 \text{ for all }0\leq r<1.\label{eqn:ave}
\end{align}
Weaker version of the maximum principle (Comparison principle): Consider the Laplace problem \eqref{eqn:le} with piecewise continuous boundary data $f$ and $\Tilde{f}$. Let the corresponding solutions be $w$ and $\Tilde{w}$. Then we have
\begin{align*}
    f\leq \Tilde{f} \text{ on }\partial B\setminus D_f\implies w\leq \Tilde{w} \text{ on }B.
\end{align*}

\section{Classical Schwarz method for overlapping balls}\label{sec:4}
We consider the Laplace problem 
\begin{align}\label{eqn:led}
    \Delta w&=0 \text{ in }\Omega=B_1\cup B_2,\\
    w&=f \text{ on }\partial\Omega,\nonumber
\end{align}
in a domain $\Omega$ which is made of two intersecting balls. Without loss of generality, we take one ball with center at $(0,0,0)$ of radius $1$, denoted by $B_1=B_1(0,0,0)$ and another ball of radius $R$ and center at $(m,0,0)$ with $m>0$, i.e. $B_2=B_R(m,0,0).$ We characterize any point on the boundary of the domain by the angles $(\theta, \phi)\in [0, \pi]\times[0,2\pi).$ We specify the boundary of each balls as $\partial B_i=\Gamma_i\cup (\partial B_i\setminus \Gamma_i)$, where $\Gamma_i=\partial B_i\cap B_j, j\neq i$ is the interior part of the boundary of the ball $B_i.$ Further, we assume that the Dirichlet boundary data $f$ is continuous on $\partial B$ and the restriction on the boundary of each ball is defined by $f_i=f|_{\partial B_i\setminus \Gamma_i}.$ We consider the additive variant of the Schwarz method to approximately solve the Laplace problem. We start with the initial approximation $w_i^{(0)}\in C(\Bar{B}_i)$ such that $w_i^{(0)}=f_i$ on $\partial B_i\setminus \Gamma_i, i=1,2.$ Then the iterative schemes are given by:\\
\begin{minipage}{0.5\textwidth}
\begin{align}\label{eqn:exsch}
    \begin{cases}
        \Delta w_1^{(k)}&= 0 \text{ in }B_1,\\
          w_1^{(k)}&=f_1 \text{ on }\partial B_1\setminus \Gamma_1,\\
          w_1^{(k)}&= w_2^{(k-1)} \text{ on }\Gamma_1,   
\end{cases}
\end{align}
\end{minipage}
\begin{minipage}{0.5\textwidth}
\begin{align}
    \begin{cases}
        \Delta w_2^{(k)}&= 0 \text{ in }B_2,\\
          w_2^{(k)}&=f_2 \text{ on }\partial B_2\setminus \Gamma_2,\\
          w_2^{(k)}&= w_1^{(k-1)} \text{ on }\Gamma_2.   
\end{cases}
\end{align}
\end{minipage}
The convergence of the above scheme is determined by $w_1^{(k)}|_{\Gamma_2}$ and $w_2^{(k)}|_{\Gamma_1}.$ We need a few notations to explain the iterative scheme. We denote the solution of the Laplace problem on the ball $B_i$ by $\Delta_i^{-1}(\Tilde{w},w), i=1,2$ for the corresponding piecewise boundary data $(\Tilde{w},w)\in C(\overline{\partial B_i\setminus\Gamma_i})\times C(\overline{\Gamma}_i).$ Moreover, we can restrict the solution $\Delta_1^{-1}(\Tilde{w},w)$ on the interior boundary $\Gamma_2$ and extend this function continuously to $\overline{\Gamma}_2.$ The operator $B_1\rightarrow \overline{\Gamma}_2$ is given by $R_{\overline{\Gamma}_2}$. Similarly, we define the operator $R_{\overline{\Gamma}_1}: B_2 \rightarrow \overline{\Gamma}_1$. Therefore, the sequences $(w_1^{(k)})_{k\geq1}$ and $(w_2^{(k)})_{k\geq1}$ can be written as
\begin{align}\label{eqn:ive}
w_1^{(k)}|_{\overline{\Gamma}_2}=R_{\overline{\Gamma}_2}\Delta_1^{-1}\left(f_1,w_2^{(k-1)}|_{\Gamma_2}\right), w_2^{(k)}|_{\overline{\Gamma}_1}=R_{\overline{\Gamma}_1}\Delta_2^{-1}\left(f_2,w_1^{(k-1)}|_{\Gamma_2}\right).
\end{align}
In addition, we introduce Dirichlet to Dirichlet map $\Lambda_1: C(\overline{\Gamma}_1)\rightarrow C(\overline{\Gamma}_2)$ by $\Lambda_1 w=R_{\overline{\Gamma}_2}\Delta_1^{-1}(0,w)$. In the same way, we define another Dirichlet to Dirichlet map $\Lambda_2: C(\overline{\Gamma}_2)\rightarrow C(\overline{\Gamma}_1)$ by $\Lambda_2 w=R_{\overline{\Gamma}_1}\Delta_2^{-1}(0,w)$. Hence, we can derive
\begin{align}\label{eqn:ime}
    \begin{pmatrix}
w_1^{(k)}|_{\overline{\Gamma}_2}\\
w_2^{(k)}|_{\overline{\Gamma}_1}
\end{pmatrix}=\begin{pmatrix}
    0 & \Lambda_1\\
    \Lambda_2 &0
\end{pmatrix}\begin{pmatrix}
w_1^{(k-1)}|_{\overline{\Gamma}_2}\\
w_2^{(k-1)}|_{\overline{\Gamma}_1}
\end{pmatrix}+\begin{pmatrix}
    R_{\overline{\Gamma}_2}\Delta_1^{-1}(f_1, 0)\\
    R_{\overline{\Gamma}_1}\Delta_2^{-1}(f_2, 0)
\end{pmatrix}.
\end{align}

\section{Schwarz-spherical domain decomposition method}\label{sec:5}
In this section, we introduce an inexact version of the classical Schwarz method. We consider the approximate boundary data on $\partial B_i, i=1,2$ given by the finite sum of the spherical harmonics as expressed in \eqref{eqn:bdata}. Therefore, the respective space for the inexact boundary data is denoted by
\begin{align}\label{eqn:bsp}
    \mathcal{S}_N^1={\textrm span}\{Y_n^m\; | \; 0\leq n \leq N, -n \leq m \leq n\}\subset L^2(\partial B_1).
\end{align}
We can parametrize the the ball $B_2$ as $(r,\xi)\rightarrow (m,0,0)+r\xi,$ where $0\leq r <R$ and $|\xi|=1.$ The corresponding solution space for the boundary data of the ball $B_2$ is $\mathcal{S}_N^2.$ We use the notation $P^i_N : L^2(\partial B_i)\rightarrow \mathcal{S}_N^i$ to describe the $L^2$-orthogonal projection of the Dirichlet boundary data. We have an option to choose different discretization parameters $N$ for different balls, and we can also vary $N$ for each iteration. However, in this work, we choose the same discretization parameter $N$ for two balls and fix $N$ for every iteration. Therefore, the inexact variant of the classical Schwarz iterations is given by: choose the initial approximation $w_i^{(0)}\in C(\overline{B_i})$ such that $w_i^{(0)}|_{\partial\Omega_i\setminus \Gamma_i}=f_i, i=1,2$ and have to determine for $k\geq 1$
\begin{minipage}{0.5\textwidth}
  \begin{align}\label{eqn:inexsch}
  \begin{cases}
       \Delta w_1^{(k)} &=0 \text{ in } B_1,\\
    w_1^{(k)} & = P^1_N(f_1, w_2^{(k-1)}) \text{ on }\partial B_1.
  \end{cases}
\end{align}  
\end{minipage}
\begin{minipage}{0.5\textwidth}
    \begin{align}
    \begin{cases}
        \Delta w_2^{(k)} &=0 \text{ in } B_2,\\
    w_2^{(k)}&=P^2_N(f_2, w_1^{(k-1)}) \text{ on }\partial B_2.
    \end{cases}
\end{align}
\end{minipage}
We can derive the solution in a ball once we can determine $P_N^1:={\textrm span}\{Y_n^m(\xi)\; | \; 0\leq n \leq N, -n \leq m \leq n, |\xi|=1\}$ as the solution in the ball $B_1$ 
\begin{align*}
    w_1^{(k)}=\sum_{n=0}^{N}\sum_{m=-n}^{n} a_{nm} r^n Y_n^m(\xi),
\end{align*}
cf. \eqref{eqn:se}. We can apply the same argument for $P_N^2$ and $w_2^{(k)}.$ We have to be careful about the fact that the Laplace problem \eqref{eqn:led} is discretized here. Since we approximate the solution space by a finite-dimensional space $\mathcal{S}_N^i, i=1,2.$ Therefore, other than the rate of convergence of the iterations, we have another crucial issue to address, that is the accuracy of the discrete solution. We can write the total error in $w_1^{(k)}$ as $e_1^{(k)}=(w|_{B_1}-w_1^{(k)})=(w|_{B_1}-w_1^{(\infty)})+(w_1^{(\infty)}-w_1^{(k)}):=e_{1,discr}+e_{1,iter},$ where we assume $w_1^{(\infty)}=\lim_{k\to \infty} w_1^{(k)}$. In this paper, we only study the iteration error $e_{1,iter}.$ 

We formulate \eqref{eqn:inexsch} analogous to $\eqref{eqn:ive}-\eqref{eqn:ime}$ for the convergence analysis of the Schwarz-spherical domain decomposition method. We consider the Dirichlet to Dirichlet mappings
\begin{align}
    \Lambda_{1,N}w:=R_{\Bar{\Gamma_2}}\Delta_1^{-1}P_N^1(0,w), \; \Lambda_{2,N}w:=R_{\Bar{\Gamma_1}}\Delta_2^{-1}P_N^2(0,w).
\end{align}
Therefore, the Schwarz-spherical method is given by
\begin{align}
     \begin{pmatrix}
w_1^{(k)}|_{\overline{\Gamma}_2}\\
w_2^{(k)}|_{\overline{\Gamma}_1}
\end{pmatrix}=\begin{pmatrix}
    0 & \Lambda_{1,N}\\
    \Lambda_{2,N} &0
\end{pmatrix}\begin{pmatrix}
w_1^{(k-1)}|_{\overline{\Gamma}_2}\\
w_2^{(k-1)}|_{\overline{\Gamma}_1}
\end{pmatrix}+\begin{pmatrix}
    R_{\overline{\Gamma}_2}\Delta_1^{-1}P_N^1(f_1, 0)\\
    R_{\overline{\Gamma}_1}\Delta_2^{-1}P_N^2(f_2, 0)
\end{pmatrix}.
\end{align}
Convergence properties of the Schwarz-spherical domain decomposition method is determined by the operator $\Lambda_N:=\begin{pmatrix}
    0 & \Lambda_{1,N}\\
    \Lambda_{2,N} &0
\end{pmatrix}$.

\subsection{A variant of the maximum principle}
We derive a different version of the maximum principle for the operator $\Delta_1^{-1}P_N^1$ here. We denote the space of piece-wise continuous functions on $\partial B_1$ by $\mathcal{C}_D \subset L^2(\partial B_1)$. We consider the space $\mathcal{C}_D$ as the domain of the projection operator to avoid technical details. We know that for the boundary data $f\in\mathcal{C}_D:$ (see section \ref{sec:2} for details)
\begin{align*}
    w=\Delta_1^{-1}P_N^1 f \text{ iff } w(x):=w(r\eta)=\frac{1}{4\pi}\int_{|\xi|=1}K(r\eta, \xi) P_N^1 (\xi)\; d\xi \text{ for all } x \in B_1.
\end{align*}
We can shift the projection operator and derive the following lemma. The explicit form of the exact Poisson kernel is given in \eqref{eqn:pk}. 
\begin{lemma}\label{lemma:fmv}
    For the boundary data $f\in \mathcal{C}_D,$ we have
    \begin{align}
        w=\Delta^{-1}P^1_N f \text{ iff } w(x):= w(r\eta)=\frac{1}{4\pi}\int_{|\xi|=1}K_N(r\eta, \xi) f(\xi)\; d\xi, \text{ for all  } x \in B_1
    \end{align}
    with the truncated kernel $K_N(r\eta, \xi)=\sum_{n=0}^{N}r^n (2n+1)L_n(\eta\cdot\xi).$
\end{lemma}
\begin{proof}
    Let $x=r\eta\in B_1$, where $0\leq r <1$ and $\eta\in \partial B_1$. Now the series $\sum_{n=0}^{\infty} (2n+1)r^nL_n(\eta\cdot\xi)$ converges uniformly to $\frac{1-r^2}{(1-2r(\eta.\xi)+r^2)^{\frac{3}{2}}}$ for $r<1$ as $|\eta\cdot\xi|=1.$ Therefore the $L^2$-orthogonal projection of the kernel on the finite dimensional space $\mathcal{S}_N^1$ (cf. \eqref{eqn:bsp}) is represented by $K_N(r\eta, \xi):=P_N^1K(r\eta, \xi)=\sum_{n=0}^{N} (2n+1)r^nL_n(\eta\cdot\xi).$ Consequently, we get
    \begin{align*}
       w(x)&=\frac{1}{4\pi}\int_{|\xi|=1}K(r\eta, \xi) P_N^1 f(\xi)\; d\xi=\frac{1}{4\pi}\int_{|\xi|=1}P_N^1 K(r\eta, \xi) f(\xi)\; d\xi\\
       &=\frac{1}{4\pi}\int_{|\xi|=1}K_N(r\eta, \xi)f(\xi)\; d\xi. 
    \end{align*}
\end{proof}
Now a necessary condition for the maximum principle \eqref{eqn:mp} is the positivity of the kernel $K(r\eta, \xi)\geq 0,$ for all $x=r\eta\in B_1.$ However, we do not expect the same  sign property from the truncated kernel $K_N$. Since the Legendre polynomials oscillate and the oscillation increases as $n$ grows. Therefore, when $r \approx 1$ the damping $r^n$ is weak and the high oscillation is not suppressed. As a result, the kernel $K_N$ changes sign as an instance of the Gibbs Phenomena. We know $x=r\eta\in B_1$, where $0\leq r<1$ and $\eta=(\sin{\theta}\cos{\phi}, \sin{\theta}\sin{\phi}, \cos{\theta})\in \partial B_1.$ We calculate $\eta\cdot\xi:=\cos{\psi},$ such that $0\leq \psi \leq \pi$, this yields 
\begin{align}\label{eqn:eqtk}
     \tilde{K}_N(r, \psi):=\sum_{n=0}^{N}r^n(2n+1)L_n(\cos{\psi})=K_N(r\eta, \xi), 0\leq \psi \leq \pi.
\end{align}
We illustrate $\Tilde{K}_N(r, \psi)$ for $N=2,10,30,100$ in the Figure \ref{fig:1} and Figure \ref{fig:2}.
\begin{figure}[h!]
    \includegraphics[width=0.5\linewidth]{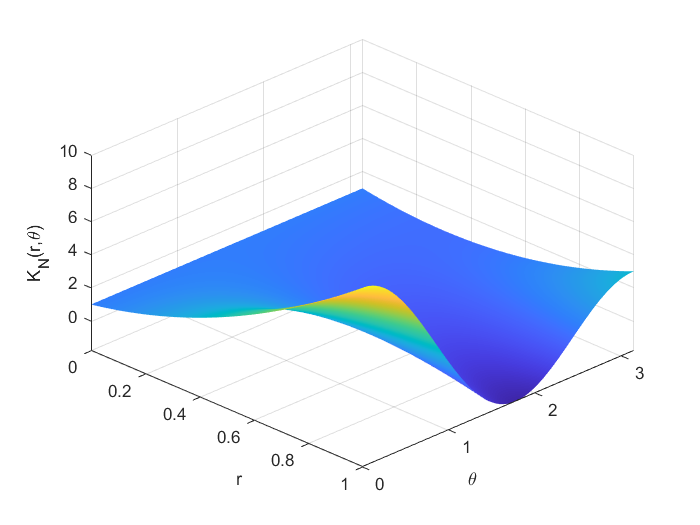}
    \includegraphics[width=0.5\linewidth]{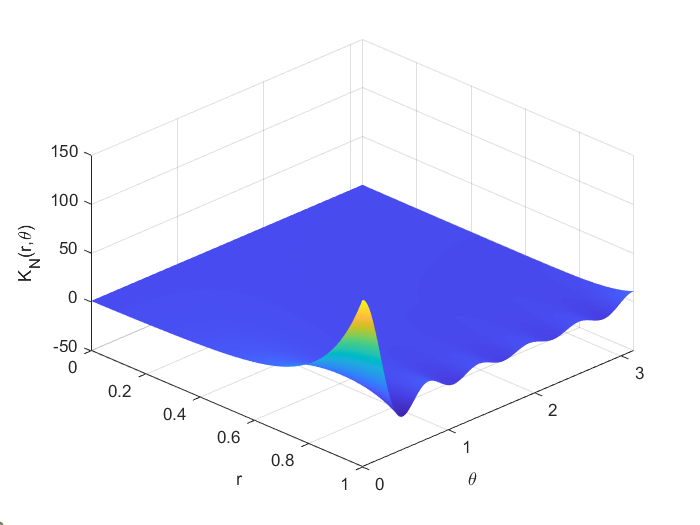}
    \caption{Projected Kernel $K_N(x,(0,0,1))$ for $N=2$(left) and $N=10$ (right) }
    \label{fig:1}
\end{figure}

\begin{figure}[h!]
    \includegraphics[width=0.5\linewidth]{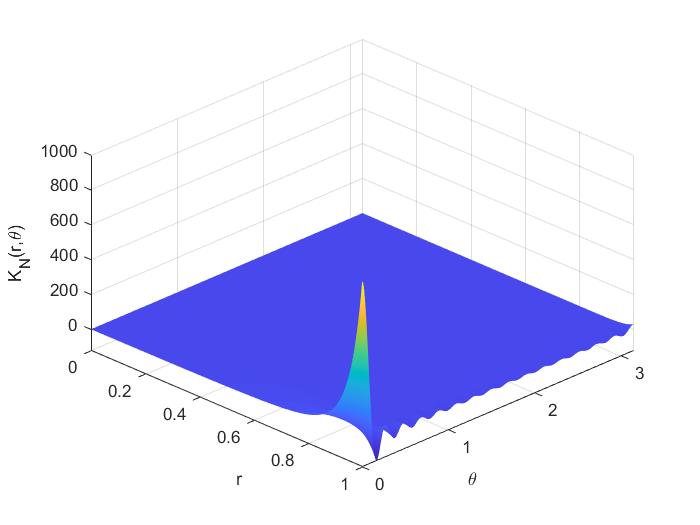}
    \includegraphics[width=0.5\linewidth]{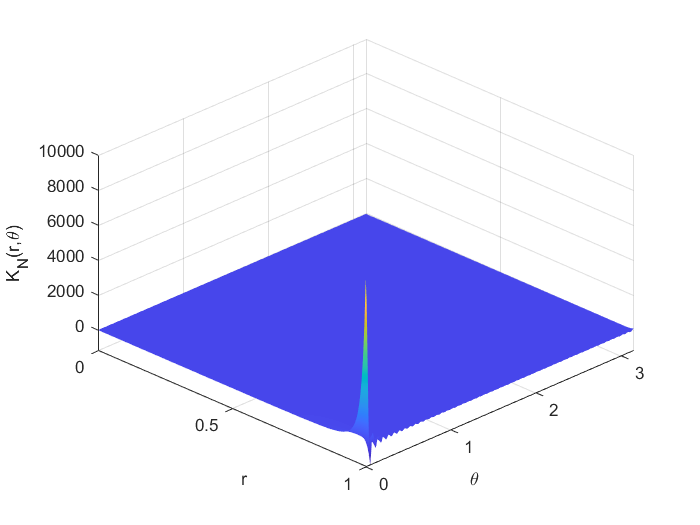}
    \caption{Projected Kernel $K_N$ for $N=30$(left) and $N=100$ (right)}
    \label{fig:2}
\end{figure}

\begin{thm}\label{thm:main}
The projected Poisson kernel satisfies: for all $N\geq 4$ odd, 
\begin{align}\label{eqn:mrs}
\begin{split}
     \tilde{K}_N(r, \psi)&=\sum_{n=0}^{N}r^n(2n+1)L_n(\cos{\psi}) \geq 0 \text{ for all } \psi\in [0,\pi] \text{ with }r\leq r_N,\\
    \text{ with } r_N&=\left(1-\frac{2\ln{(2(N+1))}}{N+1}\right).
\end{split}
\end{align}
\end{thm}

\begin{proof}
We have to find $r_N=\arg \max\{ r \,|\, \Tilde{K}_N(r, \psi)\geq 0, \forall \psi\in [0, \pi] \}$. From numerical experiments, we have observed that $r_N$ coincides with
\begin{align}
    r_N^*=\arg\max\{r \; | \; S_N(r,\psi)=\sum_{n=0}^{N}(-1)^n r^n(2n+1)\sin^{2n}{\left(\frac{\psi}{2}\right)}\geq 0, \forall \; \psi\in [0,\pi]\},
\end{align}
 for $N$ odd.
Let $x=r\sin^{2}{\frac{\psi}{2}}\in [0, 1)$. We compute the finite sum for odd $N$ and obtain
\begin{align*}
    S_N(x)&=\sum_{n=0}^{N}(-1)^n(2n+1)x^n\\
    &=\frac{1-x-(2N+3)x^{N+1}-(2N+1)x^{N+2}}{(1+x)^2}.
\end{align*}
We insert the value of $x$ and obtain
\begin{align*}
    S_N(r,\psi)=\frac{1-r\sin^2{\left(\frac{\psi}{2}\right)}-(2N+3)r^{N+1}\sin^{2N+2}{\left(\frac{\psi}{2}\right)}-(2N+1)r^{N+2}\sin^{2N+4}{\left(\frac{\psi}{2}\right)}}{\left(1+r\sin^2{\left(\frac{\psi}{2}\right)}\right)^2}.
\end{align*}
We differentiate $S_N(r,\psi)$ partially with respect to $\psi$ and arrive at
\begin{align*}
    \frac{\partial S_N}{\partial \psi}&=\frac{\cos{\left(\frac{\psi}{2}\right)}}{\left(1+r\sin^2{\left(\frac{\psi}{2}\right)}\right)^3}\left[-3r\sin{\left(\frac{\psi}{2}\right)}+r^2\sin^3{\left(\frac{\psi}{2}\right)}-(2N^2+3N+1)r^{N+1}\sin^{2N+1}{\left(\frac{\psi}{2}\right)}\right.\\
    &\left.-(4N^2+4N-3)r^{N+2}\sin^{2N+3}{\left(\frac{\psi}{2}\right)}-(2N^2+N)r^{N+3}\sin^{2N+5}{\left(\frac{\psi}{2}\right)}\right]\\
    &\leq 0.
\end{align*}
Therefore, the function $S_N(r, \psi)$ is decreasing with respect to $\psi.$ Hence, the global minimum value of $S_N$ with respect to $\psi$ will be attained at the endpoint. So, we can write
\begin{align*}
    \min_{\psi\in [0, \pi]} S_N(r,\psi)=S_N(r,\pi)=\sum_{n=0}^{N}(-1)^n(2n+1)r^n.
\end{align*}
Next, we would like to find the global minimum value of $\tilde{K}_N(r,\psi)$ with respect to $\psi$ for $\psi\in [0,\pi]$. We calculate
\begin{align*}
   \frac{\partial \tilde{K}_N}{\partial \psi}&=-\sum_{n=0}^{N}r^n(2n+1)(\sin{\psi})L'_n(\cos{\psi}),\\
   \frac{\partial^2 \tilde{K}_N}{\partial \psi^2}&=-\sum_{n=0}^{N}r^n(2n+1)[\cos{\psi}L'_n(\cos{\psi})-\sin^2{\psi}L''_n(\cos{\psi})].
\end{align*}
Now, $\partial_\psi \tilde{K}_N(r,\psi)=0$ gives $\psi=0, \pi $ or $L'_n(\cos{\psi})=0.$ 
Computing the second derivatives at the point $\psi=0$ gives $\restr{\frac{\partial^2 \tilde{K}_N}{\partial \psi^2}}{\psi=0}=-\sum_{n=0}^{N}r^n(2n+1)L'_n(1)=-\sum_{n=0}^{N}r^n(2n+1)\frac{n(n+1)}{2}\leq 0$. Therefore, $\tilde{K}_N(r,\psi)$ attains its maximum value with respect to $\psi$ at $\psi=0.$ Next, we compute the value at $\psi=\pi$, and we get
\begin{align*}
    \restr{\frac{\partial^2 \tilde{K}_N}{\partial \psi^2}}{\psi=\pi}&=\sum_{n=0}^{N}r^n(2n+1)L'_n(-1)=\frac{1}{2}\sum_{n=1}^{N}(-1)^{n+1}n(n+1)(2n+1)r^n.
\end{align*}
We need to mention the fact that $N$ is odd now plays an important role in determining the sign of the second derivative. We chose $a_n=n(n+1)(2n+1)r^n,$ $n\in \mathbb{N}$ and write the expression as
\begin{align*}
    \restr{\frac{\partial^2 \tilde{K}_N}{\partial \psi^2}}{\psi=\pi}=\frac{1}{2}\left[(a_1-a_2)+(a_3-a_4)+. . .+(a_{N-2}-a_{N-1})+a_{N}\right].
\end{align*}
Therefore, $\restr{\frac{\partial^2 \tilde{K}_N}{\partial \psi^2}}{\psi=\pi}\geq 0$ if and only of $\frac{a_{n+1}}{a_n}\leq 1,$ for all $1\leq n \leq N.$ Then we compute
\begin{align*}
    \frac{a_{n+1}}{a_{n}}=\frac{(2n+3)(n+2)}{n(2n+1)}r\leq 1
\end{align*}
iff $r \leq \frac{n(2n+1)}{(2n+3)(n+2)}$. Therefore, we prove that $\tilde{K}_N(r,\psi)$ attains global minimum value with respect to $\psi$ at $\psi=\pi$ for $r\leq \frac{n(2n+1)}{(2n+3)(n+2)}=r(N).$ Consequently, we deduce that
\begin{align}
    \min_{\psi\in [0,\pi]}\tilde{K}_N(r,\psi)=\tilde{K}_N(r,\pi)=S_N(r,\pi)=\min_{\psi\in[0,\pi]}S_N(r,\psi) \text{ for all } 0\leq r <r(N).
\end{align}
We now have to check the positivity of $K_N(r,\pi)$. The expression of $K_N(r,\pi)$ is given by  
\begin{align*}
    \tilde{K}_N(r,\pi)&=\sum_{n=0}^{N}(-1)^n(2n+1)r^n\\
    &=\frac{1-r-(2N+3)r^{N+1}-(2N+1)r^{N+2}}{(1+r)^2}.
\end{align*}
The denominator is strictly positive for $0\leq r <1.$ We have to analyze the sign of the numerator. We start with
\begin{align*}
    1-r-(2N+3)r^{N+1}-(2N+1)r^{N+2}&\geq 1-r-4(N+1)r^{N+1}=y-4x(1-y)^x
\end{align*}
with $y=1-r$ and $x=N+1.$ We have to find appropriate $y=y(x)$ such that
\begin{align}
    y-4x(1-y)^x\geq 0
\end{align}
holds. We use the Lambert function to achieve our goal. The Lambert function $W(x)$ is a unique solution of the equation $W(x)e^{W(x)}=x$ for $x\geq 0.$ We take $\tilde{y}:=\frac{W(4x^2)}{x}$. Therefore, we get $W(4x^2)=x\tilde{y}$. We know $W(4x^2)$ satisfies the relation $W(4x^2) e^{W(4x^2)}=4x^2$. Therefore, we obtain $\tilde{y} e^{x\tilde{y}}=4x.$ We can express $\tilde{y}$ as 
\begin{align*}
    \tilde{y}=4xe^{-x\tilde{y}}.
\end{align*}
We know from \cite{hassani2012approximation} that $W(x)<\ln{x}$ for $x>e.$ Therefore, we have $\tilde{y}=\frac{W(4x^2)}{x}\leq \frac{2\ln{(2x)}}{x}$. For all $N\geq 4,$ we get $0<\tilde{y}<1.$ Then, we use the inequality $\ln{(1-\tilde{y})}\leq -\tilde{y}$ to derive
\begin{align*}
    4x(1-\tilde{y})^x \leq 4xe^{-x\tilde{y}}=\tilde{y}.
\end{align*}
Therefore, $\tilde{y}-4x(1-\tilde{y})^x\geq 0$ holds for $\tilde{y}=\frac{W(4x^2)}{x}.$ The function $t\mapsto t-4x(1-t)^x$ is monotonically increasing on $[0,1).$ Consequently, we have there exists $y_0$ with $\tilde{y}\leq y_0<1$ such that
\begin{align}\label{eqn:ey}
     y-4x(1-y)^x\geq 0 \text{ for all } y\in [y_0,1).
\end{align}
We have for all $N\geq 4$, $\tilde{y}=\frac{W(4x^2)}{x}\leq \frac{2\ln{(2x)}}{x}:=y_0.$ We now insert $y=(1-r)$ and $x=N+1$ in \eqref{eqn:ey} to deduce that
\begin{align*}
    1-r-4(N+1)r^{N+1}\geq 0 \text{ if }r\leq 1-\frac{2\ln{(2(N+1))}}{N+1}.
\end{align*}
This completes the proof of the theorem.
\end{proof}

We now want to determine $r_{num}$, where $r_{num}=\arg \max\{r \; |\; \Tilde{K}_N(r, \psi)=\sum_{n=0}^{N} (2n+1)r^nL_n(\cos{\psi})\geq 0 \text{ for all } \psi\in [0, \pi] \}$. We choose $\delta_{num}=1-r_{num}$. We plotted the changes of $1/\delta_{num}$ for the different values of $N$ in the Figure \ref{fig:Nvsdel}.
\begin{center}
\begin{tabular}{ |c|c|c|c|c|c|c|c|c|c|c|} 
 \hline
 $N$ & 5 & 10 & 20 & 30  &50  &70 &80 &90 &100\\
 \hline
 $r_{num}$ &0.52994 &0.6996 &0.78942 &0.82934 &0.87924  &0.8992 & 0.90918 & 0.91916 &0.91916\\
 \hline
\end{tabular}
\end{center}
\begin{figure}[H]
    \centering
    \includegraphics[width=0.8\linewidth]{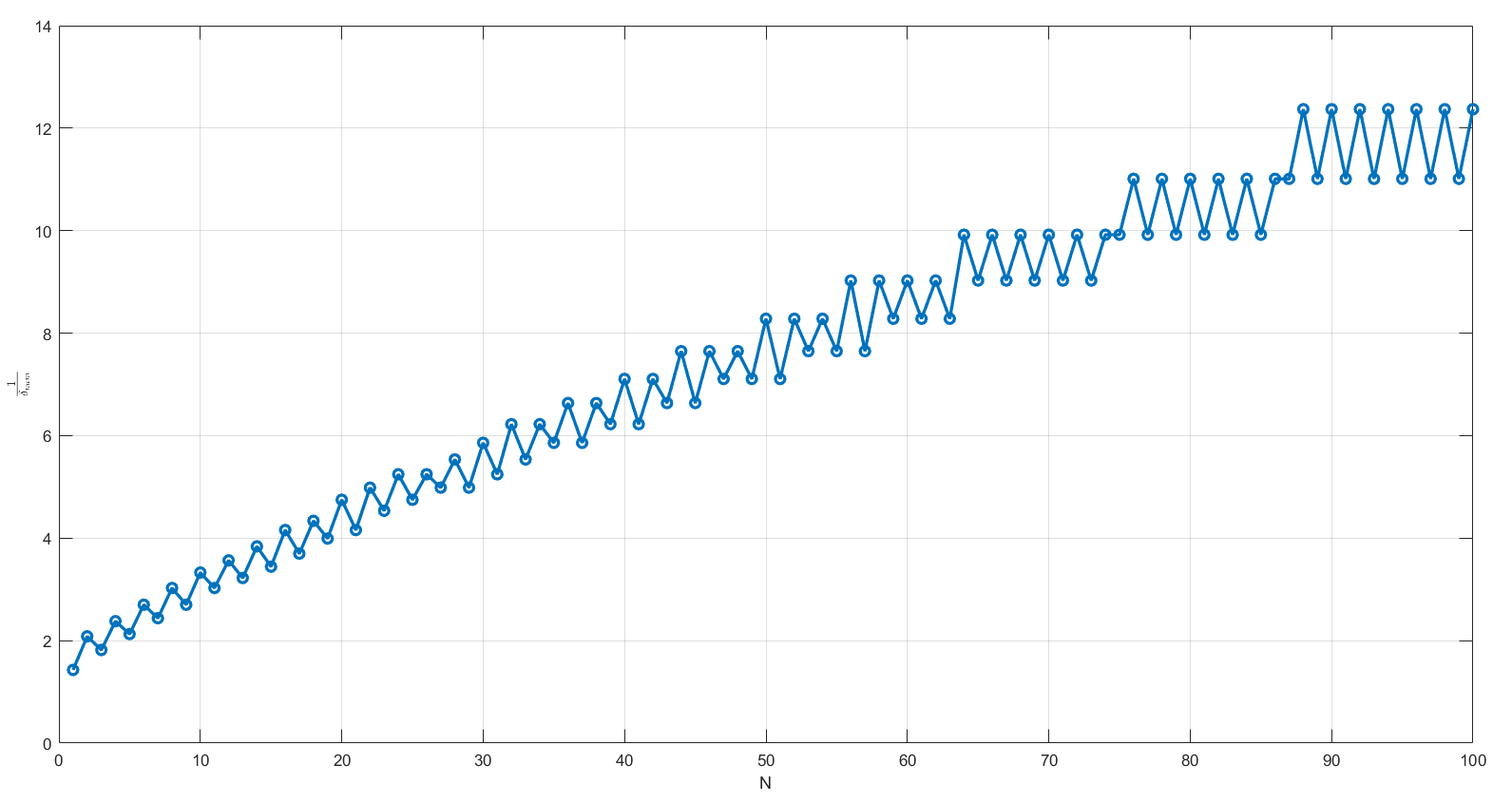}
     \caption{Plot of $\frac{1}{\delta_{num}}$ vs $N$ for up to $N=100.$}
    \label{fig:Nvsdel}
\end{figure}
We establish in Theorem \ref{thm:main} that $\tilde{K}_N(r,\psi)\geq 0$ for all $\psi\in [0,\pi]$ and $r\leq (1-\delta_{th}(N))$ with $\delta_{th}(N)=\frac{2\ln{2(N+1)}}{(N+1)}$. Clearly, we can see $\delta_{th}(N)\geq \delta_{num}(N)$ holds. We compare the value of $\frac{1}{\delta_{num}(N)}$ and $\frac{1}{\delta_{th}(N)}$ for discrete $N\in [4,100]$ and the corresponding values are plotted in the 
\begin{figure}[H]
    \centering
    \includegraphics[height=7.5cm, width=9.6cm]{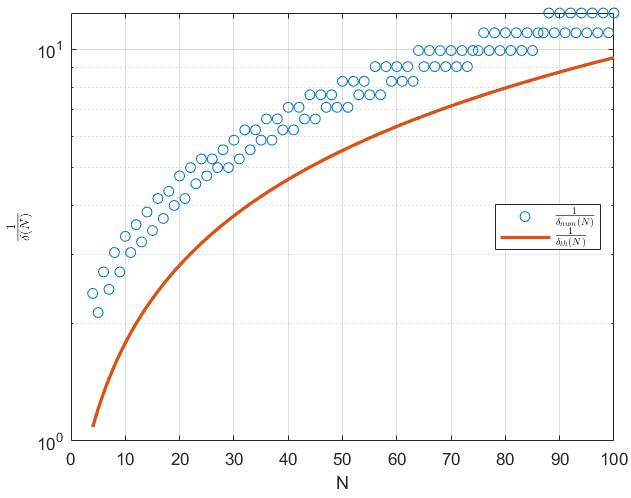}
    \caption{Plot of $\frac{1}{\delta_{num}}$ and $\frac{1}{\delta_{th}}$ vs $N$}
    \label{fig:both}
\end{figure}
Figure \ref{fig:both}. We also plotted the quotient $q(N)=\frac{\delta_{num}(N)}{\delta_{th}(N)}$ for $N\in [4,100]$ and is shown in the Figure \ref{fig:qut}.
\begin{figure}[H]
    \centering
    \includegraphics[height=6.5cm, width=10cm]{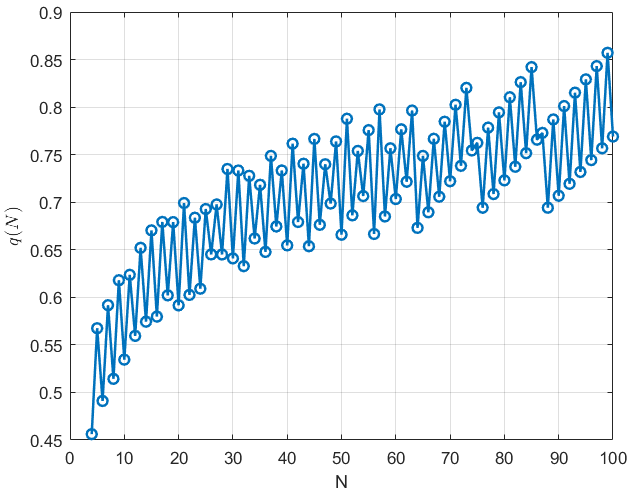}
    \caption{Plot of $q(N)=\frac{\delta_{num}(N)}{\delta_{th}(N)}$ vs $N$ }
    \label{fig:qut}
\end{figure}
We derive a new version of the maximum principle for the case when the Dirichlet boundary data is projected into the finite-dimensional space $\mathcal{S}_N^1$, explicitly given by \eqref{eqn:bsp}, by using the result of Theorem \ref{thm:main}. For this, we introduce a new notation
\begin{align*}
    B_1^{+}(N):=\{r\eta\in B_1 \; \big| \; 0\leq r\leq r_N, |\eta|=1 \} \text{ for } N\geq 4 \text{ odd }
\end{align*}
to describe the subdomain of $B_1(N)$ where the projected Poisson kernel $\tilde{K}_N$ is positive. Therefore, the new version of the maximum principle is (the fundamental maximum principle is stated in $\eqref{eqn:mp}-\eqref{eqn:ave}$) as follows:
\begin{thm}\label{thm:nmp}
    For every odd $N\geq 4$, the following properties holds 
    \begin{align}
       K_N(r\eta, \xi)&\geq 0 \text{ for all } r\eta \in B_1^{+}(N), \xi\in \partial B_1, \label{eqn:ptk}\\
      \frac{1}{4\pi}\int_{|\xi|=1}K_N(r\eta, \xi)\; d\xi &=1 \text{ for all }  0\leq r<1, \label{eqn:atk}\\
       \min_{\xi\in \partial B_1} f(\xi) \leq (\Delta^{-1}_1P_N^1f)(r\eta, \xi)&\leq \max_{\xi\in \partial B_1} f(\xi), \text{ for all }f\in \mathcal{C}_D \text{ and }r\eta\in B_1^{+}(N), \label{eqn:ulbtk}
       \end{align}
       where $K_N(r\eta, \xi)$ is explicitly given in the Lemma \ref{lemma:fmv}.
\end{thm}
\begin{proof}
    We have from \eqref{eqn:eqtk} that $K_N(r\eta, \xi)=\tilde{K}_N(r,\psi),$ for all $(r,\psi)\in [0,1)\times [0,\pi].$ Accordingly, the result \eqref{eqn:ptk} follows from Theorem \ref{thm:main}. We know that $K_N=P_N^1K$ and therefore
    \begin{align*}
        \frac{1}{4\pi}\int_{|\xi|=1}K_N(r\eta, \xi)\; d\xi=\frac{1}{4\pi}\int_{|\xi|=1}K(r\eta, \xi)P_N^1(1)\; d\xi=\frac{1}{4\pi}\int_{|\xi|=1}K(r\eta, \xi)\; d\xi=1
    \end{align*}
   by \eqref{eqn:ave}. This proves \eqref{eqn:atk}. From Lemma \ref{lemma:fmv}, we obtain
   \begin{align*}
       \Delta^{-1}P_N^1f(x):=\Delta^{-1}P_N^1f(r\eta)=\frac{1}{4\pi}\int_{|\xi|=1}K_N(r\eta, \xi)f(\xi)\; d\xi, \text{ for all }x\in B_1.
   \end{align*}
   Therefore, for all $r\eta, r\xi\in B_1^{+}(N),$ we have $K_N(r\eta, \xi)\geq0.$ Together with \eqref{eqn:atk} this implies \eqref{eqn:ulbtk}.
\end{proof}

\begin{rem}
    Theorem \ref{thm:nmp} is new.
\end{rem}

We apply Theorem \ref{thm:nmp} to analyze the convergence of the inexact domain decomposition methods in $3$D. We have to calculate
\begin{align*}
    \|\Lambda_{1,N}\|_{\infty}=\max_{w\in C(\overline{\Gamma}_1)}\frac{\|\Lambda_{1,N}w\|_{L^\infty(\overline{\Gamma}_2)}}{\|w\|_{L^\infty(\overline{\Gamma}_1)}}=\max_{w\in C(\overline{\Gamma}_1)}\frac{\|R_{\overline{\Gamma}_2}\Delta^{-1}_1P_N^1(0,w)\|_{L^\infty(\overline{\Gamma}_2)}}{\|w\|_{L^\infty(\overline{\Gamma}_1)}}
\end{align*}
We divide $\overline{\Gamma}_2$ into two parts. $\overline{\Gamma}_2=\Gamma_2^{(N)}\bigcup (\overline{\Gamma}_2\setminus \Gamma_2^{(N)})$. Therefore, we can write
\begin{align*}
\|R_{\overline{\Gamma}_2}\Delta^{-1}_1P_N^1(0,w)\|_{L^\infty(\overline{\Gamma}_2)}=\max\left\{\|R_{\overline{\Gamma}_2}\Delta^{-1}_1P_N^1(0,w)\|_{L^\infty(\Gamma_2^{(N)})}, \|R_{\overline{\Gamma}_2}\Delta^{-1}_1P_N^1(0,w)\|_{L^\infty(\overline{\Gamma}_2\setminus\Gamma_2^{(N)})}\right\}
\end{align*}
We consider $\tilde{w}$ such that
\begin{align*}
    \tilde{w}(x)=\begin{cases}
        w(x) \text{ for }x\in \overline{\Gamma}_1\\
        0 \text{ otherwise}
    \end{cases}
\end{align*}
Let $x=r\eta \in \Gamma_2^{(N)}$, then we have from Lemma \ref{lemma:fmv} that
\begin{align*}
    \left|(\Delta^{-1}_1P_N^1(0,w))(r\eta)\right|&=\left|\frac{1}{4\pi}\int_{|\xi|=1}K_N(r\eta, \xi)\tilde{w}(\xi)\; d\xi\right|\\
    &\leq \frac{1}{4\pi}\int_{\substack{|\xi|=1\\\xi\in \overline{\Gamma}_1}}K_N(r\eta, \xi)\; d\xi \|w\|_{L^\infty(\overline{\Gamma}_1)}\\
    &\leq \frac{1}{4\pi}\left(\int_{|\xi|=1}K(r\eta, \xi)(P_N^1\chi_{\Gamma_1})(\xi)\; d\xi\right)\|w\|_{L^\infty(\overline{\Gamma}_1)}\\
    &=(\Delta^{-1}P_N^1\chi_{\Gamma_1})(\xi)\|w\|_{L^\infty(\overline{\Gamma}_1)}.
\end{align*}
We know that
\begin{align*}
    (P_N^1\chi_{\Gamma_1})(\xi):=(P_N^1\chi_{\Gamma_1})(\theta,\phi)=\sum_{n=0}^{N}\sum_{m=-n}^{n}a_{nm}Y_n^m(\theta, \phi),\\
    \text{ with } a_{nm}=\langle \chi_{\Gamma_1}, Y_n^m\rangle_{L^2(\partial B_1}=\int_{0}^{\pi}\int_{0}^{2\pi}\chi_{\Gamma_1}(\theta, \phi)\overline{Y_n^m(\theta,\phi)}\sin{\theta}\;d\phi d\theta
\end{align*}



 	\bibliographystyle{acm}
	\bibliography{cite}

\end{document}